\documentclass[11pt]{article}

\usepackage[T1]{fontenc}
\usepackage{lmodern}
\usepackage{amsmath,amssymb,amsthm}
\usepackage{enumitem}
\usepackage{microtype}
\usepackage[a4paper,margin=1in]{geometry}
\usepackage[hidelinks]{hyperref}

\newtheorem{theorem}{Theorem}[section]
\newtheorem{lemma}[theorem]{Lemma}

\newtheorem{corollary}[theorem]{Corollary}
\theoremstyle{definition}
\newtheorem{definition}[theorem]{Definition}
\newtheorem{remark}[theorem]{Remark}

\newcommand{\topm}{\leq_{\mathrm{top}}}
\newcommand{\topeq}{\equiv_{\mathrm{top}}}
\newcommand{\rootm}{\leq_{\mathrm r}}
\newcommand{\plantm}{\leq_*}
\newcommand{\mle}{\leq_M}
\newcommand{\meq}{\equiv_M}
\newcommand{\qeq}{\equiv_Q}
\newcommand{\rk}{\operatorname{rk}}
\newcommand{\Br}{\operatorname{Br}}
\newcommand{\Hull}{\operatorname{Hull}}
\newcommand{\Aut}{\operatorname{Aut}}

\title{A cardinal trichotomy for topological equivalence classes of countable trees\thanks{This work is supported by the National Natural Science Foundation of China (Nos.~12371348 and 12201258) and the High-Quality Science and Technology Cultivation Project of Jiangsu Normal University (No.~JSNUGZL2026069).}}
\author{Qi Wu, Yong Lu\thanks{Corresponding author.}\\[2pt]
\small School of Mathematics and Statistics, Jiangsu Normal University,\\[-1pt]
\small Xuzhou, Jiangsu 221116, People's Republic of China\\[-1pt]
\small Emails:~\texttt{wuqimath@163.com, luyong@jsnu.edu.cn}}
\date{}

\begin{document}
\maketitle

\begin{abstract}
For a tree $T$, let $[T]$ be the set of isomorphism classes of trees that are mutually topological minors of $T$. Bruno and Szeptycki proved that every locally finite tree satisfies $|[T]|\in\{1,2^{\aleph_0}\}$ and that every tree with a ray containing infinitely many vertices of degree at least $3$ has at least $2^{\aleph_0}$ topological twins. Hence equality holds in the latter case when the tree is countable. We treat the complementary countable case, without any bound on the degrees. For every such tree $T$, we prove that $|[T]|\in\{1,\aleph_0,2^{\aleph_0}\}$ and give a well-founded recursive description based on Schmidt rank. The proof uses a finite canonical subtree extracted from the subtree generated by the branching vertices and a counting theorem for countable multisets over a well-quasi-order. The same trichotomy follows for every countable tree and hence for every tree all of whose vertices have countable degree.
\end{abstract}



\section{Introduction}

Throughout, graphs are simple and undirected, and countable means finite or countably infinite. A path is finite and may consist of one vertex. A tree is a connected acyclic graph. A ray is a one-way infinite path, and a double ray is a two-way infinite path. A graph is \emph{rayless} if it contains no ray. A leaf is a vertex of degree exactly $1$, and a branching vertex is a vertex of degree at least $3$. We call a tree \emph{large} if it has a ray containing infinitely many branching vertices, and \emph{small} otherwise. A graph is \emph{locally countable} if every vertex has countable degree.

A subdivision of a graph is obtained by replacing each edge with a finite path having at least one edge, so that distinct replacement paths are internally disjoint. For trees $T$ and $S$, write $T\topm S$ if a subdivision of $T$ is isomorphic to a subgraph of $S$, and write $T\topeq S$ if both $T\topm S$ and $S\topm T$. For a fixed tree $T$, let $[T]$ be the set of isomorphism classes of all trees $S$ with $S\topeq T$. We call such an $S$ a \emph{topological twin} of $T$. Two infinite trees may be topological twins without being isomorphic, and we ask for the possible cardinalities of $[T]$.

Related questions have been studied for ordinary embeddings, where mutually embeddable trees are usually called \emph{siblings}. Bonato and Tardif \cite{BT06} proved that a rayless tree has either one or infinitely many siblings and proposed the same alternative for all trees. Tyomkyn \cite{Tyomkyn09} proved the rooted version, and Laflamme, Pouzet, and Sauer \cite{LPS17} proved the alternative for scattered trees. The unrestricted conjecture is false: Abdi Kalow, Laflamme, Tateno, and Woodrow \cite{AKLTW23} constructed locally finite trees with any prescribed positive finite sibling number. Abdi \cite{AbdiCB23} introduced a Cantor--Bendixson rank for trees and represented a tree as a leafless core with leafy trees attached. For direct sums of chains, Abdi later proved the trichotomy $1,\aleph_0,2^{\aleph_0}$ in the countable case \cite[Theorem~4.6]{AbdiDSC25}. These works use ordinary graph embeddings for trees or order embeddings for direct sums of chains; neither settles the topological-minor problem considered here.

There are two counting questions. One may count all equivalence classes under $\topeq$, or one may fix $T$ and count the members of $[T]$. The first is global. Matthiesen \cite{Matthiesen06} proved that there are uncountably many topological types of locally finite rooted trees, and Bruno \cite{Bruno17} gave an explicit family of $\omega_1$ types. Bruno and Szeptycki \cite{BStypes22} proved that locally finite trees with countably many rays have exactly $\aleph_1$ topological types. Krill and Pitz \cite{KP24} later proved that, for every infinite cardinal $\kappa$, there are exactly $\kappa^+$ topological types of trees of size $\kappa$. These results do not determine the size of a fixed class $[T]$.

For the local problem, Bruno and Szeptycki \cite[Theorem~2]{BS22} proved that every locally finite tree $T$ satisfies $|[T]|\in\{1,2^{\aleph_0}\}$. They later proved that every tree has either one or at least countably many topological twins \cite[Theorem~1]{BS23}. Their curtailing argument also shows that every large tree has at least $2^{\aleph_0}$ topological twins \cite[Theorem~3.1]{BS23}. Thus every countable large tree has exactly continuum many topological twins. These results leave open the exact cardinality of $[T]$ for countable small trees of arbitrary degree.

Our main result is the following.

\begin{theorem}\label{thm:main}
Every countable small tree $T$ satisfies $|[T]|\in\{1,\aleph_0,2^{\aleph_0}\}$. Moreover, the value is given by a well-founded recursion on Schmidt rank.
\end{theorem}

Combining this theorem with the large-tree result gives the full countable statement.

\begin{corollary}\label{cor:countable}
Every countable tree $T$ satisfies $|[T]|\in\{1,\aleph_0,2^{\aleph_0}\}$. If $T$ is large, then $|[T]|=2^{\aleph_0}$.
\end{corollary}

All three values occur. A ray has one topological twin. Joining countably many paths of length $2$ at a common end gives a tree with $\aleph_0$ topological twins. Joining one path of each positive finite length at a common end gives a tree with $2^{\aleph_0}$ topological twins; see Corollary~\ref{cor:finite-hull}.

The proof has three layers. For a small tree, the subtree generated by the branching vertices is rayless, so Schmidt's rank \cite{Schmidt83} applies. We use rank monotonicity under topological minors from \cite[Lemma~2.1]{KP24} and the finite-kernel theory of Bonato, Bruhn, Diestel, and Spr\"ussel \cite[Section~2]{BBDS11}. Their kernel mapping, finite-kernel decomposition, and rank induction for ordinary embeddings \cite[Lemmas~3--5]{BBDS11} are the closest structural precedents for the argument below. In the present setting, subdivisions and anchors require a separate topological adaptation.

The components attached at one vertex of the resulting finite subtree form a countable multiset of planted-tree types. Nash-Williams's theorem gives the underlying better-quasi-ordering of trees by topological minors \cite{NW65}; see also K\"uhn's shorter proof \cite{Kuhn01}. The rooted tree-order preserving form used below is recorded in \cite[Section~4]{KP24}. For the order-theoretic background, Diestel studied the domination order induced on subsets of a well-quasi-order \cite{Diestel01}. Component replacement and increasing-chain constructions occur in the sibling literature, notably in \cite[Lemmas~2.2 and~2.3]{HPW24} and \cite[Sections~3 and~4]{AbdiDSC25}. The exceptional-element estimate used below is the countable case of \cite[Lemma~3.1]{KP24}, whose proof is traced there to \cite[Lemma~3.3]{BEEGHPT22}.

Our new ingredients are an exact fixed-class formula for countable occurrence multisets over an arbitrary well-quasi-order, an anchored topological decomposition of countable small trees, and the resulting fixed-tree recursion. The canonical decomposition reduces the global count to finitely many local multiset counts, and transfinite induction on Schmidt rank completes the proof. At rank $0$, every attached piece is a finite path or a ray, and the recursion becomes a criterion in terms of arm lengths and multiplicities.

Section~2 develops the branch-hull and Schmidt-rank tools, constructs the finite decomposition subtree, and compares the two rooted relations. Section~3 proves the multiset counting theorem. Section~4 establishes the canonical decomposition and the finite-core counting formula. Section~5 proves the recursion and the main results.

\section{Branch hulls, rank, and rooted models}

For vertices $x,y$ of a tree $T$, let $P_T(x,y)$ denote the unique $x$--$y$ path.

\begin{definition}\label{def:topological-model}
A \emph{topological model} of a graph $H$ in a graph $G$ consists of an injection $\phi:V(H)\to V(G)$ and, for each edge $xy\in E(H)$, a $\phi(x)$--$\phi(y)$ path in $G$. Paths assigned to distinct edges meet only at the images of common endvertices, and no internal vertex of an edge-model path belongs to $\phi(V(H))$. We write $H\topm G$ when such a model exists. If $G$ is a tree, the edge-model paths are uniquely determined by $\phi$, and we identify the model with its vertex map.
\end{definition}

\begin{definition}\label{def:anchored}
An \emph{anchored tree} is a pair $(T,A)$, where either $A=\varnothing$, or $A=\{r\}$ and $r$ is a leaf of $T$. In the second case, we write $(T,r)$ and call it a \emph{planted tree}. A topological model of $(T,A)$ in $(S,A')$ is \emph{anchor-preserving} if its vertex map sends $A$ onto $A'$. An anchored isomorphism is defined in the same way. We use $\topm$ and $\topeq$ for anchored trees with this anchor-preserving meaning. An anchored tree is called small or large according to its underlying tree. If $(C,A)$ is finite, then $\Aut(C,A)$ denotes the automorphism group of $C$ that preserves $A$.
\end{definition}

For a tree $T$, put $\Br(T)=\{v\in V(T):\deg_T(v)\geq3\}$. For a nonempty set $X\subseteq V(T)$, let $\Hull_T(X)$ be the smallest subtree of $T$ containing $X$.

\begin{definition}\label{def:branch-hull}
Let $(T,A)$ be an anchored tree. If $\Br(T)\cup A\neq\varnothing$, its \emph{branch hull} is $B(T,A)=\Hull_T(\Br(T)\cup A)$. If $A=\varnothing$ and $\Br(T)=\varnothing$, set $B(T,A)=\varnothing$.
\end{definition}

\begin{lemma}\label{lem:small-hull}
An anchored tree $(T,A)$ is small if and only if $B(T,A)$ is rayless.
\end{lemma}

\begin{proof}
Suppose that $T$ is large, and let $R$ be a ray containing branching vertices $b_0,b_1,\ldots$ in this order. The union of the paths $P_T(b_i,b_{i+1})$ contains a ray in $B(T,A)$.

Conversely, suppose that $B(T,A)$ contains a ray $R$. If $\Br(T)\neq\varnothing$, adding the anchor to $\Hull_T(\Br(T))$ adds at most one finite path. Hence a tail of $R$ lies in $\Hull_T(\Br(T))$. If this tail contained only finitely many branching vertices, then a further tail would contain none. Every vertex of that tail already has two neighbours on the ray, so it has no edge leaving the ray. Fix a vertex $x$ on this tail. Since $x\in\Hull_T(\Br(T))$, it lies on a path between two branching vertices. One of these vertices must lie in the forward component of $T-x$, but that component contains no branching vertex, a contradiction. Thus $R$ contains infinitely many branching vertices. If $\Br(T)=\varnothing$, then $B(T,A)$ is empty or consists only of the anchor, so it contains no ray.
\end{proof}

\begin{lemma}\label{lem:hull-monotone}
Let $\phi:(T,A)\topm(S,A')$ be an anchor-preserving model. If $B(T,A)\neq\varnothing$, then $\phi$ restricts to a model $B(T,A)\topm B(S,A')$.
\end{lemma}

\begin{proof}
If $v\in\Br(T)$, then the model paths of the edges incident with $v$ leave $\phi(v)$ in at least three different directions. Hence $\phi(v)\in\Br(S)$. The anchor maps to the anchor. If $x$ lies on a path between two vertices of $\Br(T)\cup A$, then $\phi(x)$ lies on the path between their images and therefore belongs to $B(S,A')$. Since $B(S,A')$ is a subtree, the edge-model paths of $B(T,A)$ also lie in it.
\end{proof}

\begin{corollary}\label{cor:smallness-downward}
If $(T,A)\topm(S,A')$ and $S$ is small, then $T$ is small.
\end{corollary}

\begin{proof}
If $T$ were large, Lemma~\ref{lem:small-hull} would give a ray in $B(T,A)$. Lemma~\ref{lem:hull-monotone} would then place a subdivision of that ray in $B(S,A')$. Every such subdivision contains a ray, contrary to Lemma~\ref{lem:small-hull}.
\end{proof}

\begin{corollary}\label{cor:smallness-invariant}
If $(T,A)\topeq(S,A')$, then $T$ is small if and only if $S$ is small. Moreover, $B(T,A)=\varnothing$ if and only if $B(S,A')=\varnothing$.
\end{corollary}

\begin{proof}
Apply Corollary~\ref{cor:smallness-downward} in both directions. The second assertion follows from Lemma~\ref{lem:hull-monotone}, again in both directions.
\end{proof}

\begin{definition}\label{def:schmidt-rank}
The \emph{Schmidt rank} of a rayless graph is defined recursively. A finite graph has rank $0$. For an ordinal $\alpha>0$, a graph $G$ has rank $\alpha$ if $\alpha$ is the least ordinal for which some finite set $X\subseteq V(G)$ has the property that every component of $G-X$ has rank below $\alpha$. We write this rank as $\rk(G)$. If $G$ has positive rank, a finite set $X\subseteq V(G)$ is \emph{rank-reducing} if every component of $G-X$ has rank below $\rk(G)$.
\end{definition}

Schmidt \cite{Schmidt83} proved that a graph has a Schmidt rank exactly when it is rayless. A topological minor of a rayless graph is rayless. Moreover, if $G$ is rayless and $H\topm G$, then $\rk(H)\leq\rk(G)$ \cite[Lemma~2.1]{KP24}. Bonato, Bruhn, Diestel, and Spr\"ussel \cite[Section~2]{BBDS11} proved that every infinite rayless graph has a unique inclusion-minimal rank-reducing set.

\begin{definition}\label{def:kernel}
For an infinite rayless graph $G$, its unique inclusion-minimal rank-reducing set is the \emph{kernel} of $G$, denoted by $K(G)$.
\end{definition}

By the uniqueness of the inclusion-minimal rank-reducing set in \cite[Section~2]{BBDS11}, the kernel is contained in every rank-reducing set. We also use the fact proved there that the kernel is nonempty when the graph is connected and has positive rank. Finally, the observation preceding Lemma~3 of \cite{BBDS11} states that if $H\subseteq G$ and $\rk(H)=\rk(G)>0$, then $K(H)\subseteq K(G)$.

The following observation is the finite-extension form of the finite-deletion property of Schmidt rank from \cite[Section~2]{BBDS11}. We include the short verification because the formulation below is the one used for planted pieces.

\begin{lemma}\label{lem:finite-extension}
Let $G$ be a rayless induced subgraph of $H$, and suppose that $V(H)\setminus V(G)$ is finite. Then $H$ is rayless and $\rk(H)=\rk(G)$.
\end{lemma}

\begin{proof}
A ray of $H$ would have a tail in $G$, so $H$ is rayless. Rank monotonicity \cite[Lemma~2.1]{KP24} gives $\rk(G)\leq\rk(H)$. If $G$ is finite, then $H$ is finite and both ranks are $0$. Suppose that $\rk(G)=\alpha>0$. Choose a finite rank-reducing set $Y$ in $G$ and put $X=V(H)\setminus V(G)$. Every component of $H-(X\cup Y)=G-Y$ has rank below $\alpha$, so the defining recursion of Schmidt rank gives $\rk(H)\leq\alpha$. Thus $\rk(H)=\rk(G)$.
\end{proof}

The next lemma is the subdivision form of the rank and kernel properties in \cite[Lemma~2.1]{KP24} and \cite[Section~2]{BBDS11}. Its kernel statement is needed because a topological model supplies a subdivision rather than an ordinary copy.

\begin{lemma}\label{lem:subdivision-kernel}
Let $R'$ be a subdivision of a connected rayless tree $R$. Then $\rk(R')=\rk(R)$. If $R$ is infinite, then $K(R')=K(R)$ under the natural identification of the original vertices.
\end{lemma}

\begin{proof}
The tree $R'$ is rayless, since suppressing the subdivision vertices on a ray of $R'$ would produce a ray of $R$. We use induction on $\alpha=\rk(R)$. The assertion is clear when $\alpha=0$. Let $\alpha>0$ and put $K=K(R)$. Every component $D$ of $R-K$ has rank below $\alpha$.

Consider a component $D'$ of $R'-K$. Suppose first that $D'$ contains an original vertex outside $K$. Its original vertices lie in one component $D$ of $R-K$. The subdivision of $D$ contained in $D'$ has rank $\rk(D)$ by induction. Since $R$ is a tree and $K$ is finite, at most $|K|$ edges join $D$ to $K$. Thus $D'$ is obtained from that subdivision by adding only finitely many subdivision vertices on these boundary edges. Lemma~\ref{lem:finite-extension} gives $\rk(D')=\rk(D)<\alpha$. If $D'$ has no original vertex outside $K$, then it is a finite path formed by subdivision vertices of an edge with both ends in $K$, and hence has rank $0$.

Thus $K$ is rank-reducing in $R'$, so $\rk(R')\leq\alpha$. Since $R\topm R'$, rank monotonicity \cite[Lemma~2.1]{KP24} gives the reverse inequality. Hence $\rk(R')=\alpha$.

Now assume that $R$ is infinite. Since $K$ is rank-reducing in $R'$, the kernel containment property from \cite[Section~2]{BBDS11} gives $K(R')\subseteq K$. Put $L=K(R')$. Every component of $R-L$ is a topological minor of a component of $R'-L$. All components of $R'-L$ have rank below $\alpha$, so rank monotonicity shows that $L$ is rank-reducing in $R$. The same property of \cite[Section~2]{BBDS11} gives $K\subseteq L$. Therefore $K(R')=K(R)$.
\end{proof}

The following result is the topological-model analogue of the kernel-mapping lemma for ordinary embeddings in \cite[Lemma~3]{BBDS11}; subdivision invariance is the additional ingredient.

\begin{lemma}\label{lem:kernel-map}
Let $R$ and $S$ be connected rayless trees of the same positive rank, and suppose that $R\topeq S$. If $\phi:R\topm S$, then $\phi(K(R))=K(S)$.
\end{lemma}

\begin{proof}
Let $R_\phi\subseteq S$ be the subdivision of $R$ supplied by the model. Lemma~\ref{lem:subdivision-kernel} gives $\rk(R_\phi)=\rk(S)$ and $K(R_\phi)=\phi(K(R))$. Since $R_\phi$ is a same-rank subgraph of $S$, the observation preceding Lemma~3 of \cite{BBDS11} yields $\phi(K(R))\subseteq K(S)$. Applying the same argument to a model of $S$ in $R$ gives $|K(S)|\leq|K(R)|$, while the displayed inclusion gives $|K(R)|\leq|K(S)|$. The kernels are finite, so $\phi(K(R))=K(S)$.
\end{proof}

\begin{definition}\label{def:core}
Let $(T,A)$ be a small anchored tree with nonempty branch hull $B=B(T,A)$. Its \emph{branch rank} is $\rho(T,A)=\rk(B)$. If $\rho(T,A)=0$, set $C(T,A)=B$; if $\rho(T,A)>0$, set $C(T,A)=\Hull_B(K(B)\cup A)$. We call $C(T,A)$ the \emph{canonical core} of $(T,A)$.
\end{definition}

The canonical core is finite. In rank $0$, the branch hull is finite. In positive rank, $K(B)\cup A$ is finite, and its hull is a finite union of finite paths.

\begin{corollary}\label{cor:core-rigidity}
Let $(T,A)$ and $(S,A')$ be small anchored trees with nonempty branch hulls, and suppose that $(T,A)\topeq(S,A')$. Then $\rho(T,A)=\rho(S,A')$. Moreover, every model $\phi:(T,A)\topm(S,A')$ restricts to an isomorphism from $C(T,A)$ onto $C(S,A')$.
\end{corollary}

\begin{proof}
Lemma~\ref{lem:hull-monotone} and rank monotonicity give equality of the branch ranks. Suppose first that the common rank is positive. Lemma~\ref{lem:kernel-map}, applied to the two branch hulls, shows that $\phi$ maps one kernel onto the other. The anchor is also preserved. Hence every source-core vertex maps into the target core. The same inclusion is immediate in rank $0$, because the core is the branch hull. A reverse model gives the opposite inequality between the finite core orders, so $\phi$ maps the source core bijectively onto the target core.

Let $xy$ be an edge of the source core. Its model path lies in the target core, which is a subtree containing $\phi(x)$ and $\phi(y)$. If the path had an internal vertex, that vertex would be the image of another source-core vertex, because the core map is bijective. This is forbidden in a topological model. Thus $\phi(x)\phi(y)$ is an edge. The restriction is therefore an isomorphism.
\end{proof}

\begin{definition}\label{def:rooted-models}
A \emph{rooted tree} is a pair $(T,r)$ consisting of a tree and a distinguished vertex. Write $x\preceq_r y$ when $x$ lies on $P_T(r,y)$. Let $(P,r)$ and $(Q,s)$ be rooted trees, and let $\phi$ be a topological model of the underlying tree $P$ in the underlying tree $Q$. The model is \emph{tree-order preserving} if $x\preceq_r y$ implies $\phi(x)\preceq_s\phi(y)$. We write $(P,r)\rootm(Q,s)$ when such a model exists; no condition that $\phi(r)=s$ is imposed. A model is \emph{root-fixing} if $\phi(r)=s$, and we write $(P,r)\plantm(Q,s)$ when a root-fixing model exists. For planted trees, $\plantm$ is the anchor-preserving relation from Definition~\ref{def:anchored}.
\end{definition}

The rooted relation in \cite[Section~4]{KP24} need not fix the root. For planted trees, the following elementary adjustment supplies the root-fixing relation required by the decomposition.

\begin{lemma}\label{lem:root-fixing}
For planted trees $(P,r)$ and $(Q,s)$, we have $(P,r)\rootm(Q,s)$ if and only if $(P,r)\plantm(Q,s)$.
\end{lemma}

\begin{proof}
A root-fixing model preserves the tree order: if $x$ lies on the source path from $r$ to $y$, then $\phi(x)$ lies on the target path from $s$ to $\phi(y)$.

Conversely, let $\phi$ be a tree-order preserving model of $P$ in $Q$, put $z=\phi(r)$, and let $u$ be the unique neighbour of $r$. If $z=s$, there is nothing to prove. Assume that $z\neq s$. Since $r\preceq_r y$ for every $y\in V(P)$, the vertex $z$ lies on every path from $s$ to a vertex image. The segment $P_Q(s,z)-z$ is disjoint from the model: no other source vertex maps to it, and an edge-model path whose ends lie below $z$ cannot enter it. Send $r$ to $s$, keep all other vertex images, and replace the model path for $ru$ by $P_Q(s,\phi(u))$. The old image $z$ becomes an internal vertex of the new path, and the added segment is disjoint from every other model path. This gives a root-fixing model.
\end{proof}

The rooted tree-order preserving topological-minor relation is a well-quasi-order; see \cite[Section~4]{KP24}, based on Nash-Williams's theorem \cite{NW65}, and also \cite{Kuhn01}. Combining this result with Lemma~\ref{lem:root-fixing} gives the following corollary.

\begin{corollary}\label{cor:planted-wqo}
The isomorphism types of countable planted trees are well-quasi-ordered by $\plantm$.
\end{corollary}

\section{Countable multisets over a well-quasi-order}

\begin{definition}\label{def:wqo-multiset}
A quasi-order $(Q,\leq)$ is a \emph{well-quasi-order}, or \emph{wqo}, if every infinite sequence $q_0,q_1,\ldots$ contains indices $i<j$ with $q_i\leq q_j$. A \emph{countable $Q$-multiset} $A$ consists of a finite or countably infinite occurrence set $\Omega_A$ and a type map $t_A:\Omega_A\to Q$. Elements of $Q$ are called \emph{exact types}. We write $A\mle B$ if there is an injection $f:\Omega_A\to\Omega_B$ such that $t_A(a)\leq t_B(f(a))$ for every $a\in\Omega_A$, and write $A\meq B$ if both $A\mle B$ and $B\mle A$. Two $Q$-multisets are \emph{isomorphic} if a bijection of their occurrence sets preserves exact types.
\end{definition}

\begin{definition}\label{def:multiset-parts}
For $q,q'\in Q$, write $q\qeq q'$ if $q\leq q'$ and $q'\leq q$. The quotient $Q/{\qeq}$, ordered by $[q]_{\qeq}\leq[q']_{\qeq}$ whenever $q\leq q'$, is a partially ordered wqo. For a countable $Q$-multiset $A$, let $I(A)=\{q\in Q:|\{a\in\Omega_A:q\leq t_A(a)\}|=\aleph_0\}$. Let $F(A)$ be the submultiset induced by the occurrences $a$ with $t_A(a)\notin I(A)$; this is the \emph{exceptional part} of $A$. For countable $Q$-multisets $A$ and $E$, let $A+E$ denote their disjoint multiset sum. If $E$ is a submultiset of $A$, then $A-E$ denotes deletion of the occurrences of $E$. A $\qeq$-class is \emph{rigid} if it contains one exact type. A subset of a poset is an \emph{antichain} if no two distinct elements are comparable, and it is \emph{cofinal} if every element of the poset lies below one of its elements.
\end{definition}

The set $I(A)$ is downward closed and is a union of $\qeq$-classes. Indeed, if $p\leq q\in I(A)$, then every occurrence above $q$ is also above $p$, so $p\in I(A)$. If $p\qeq q$, then $p$ and $q$ have the same upper occurrences; hence $p\in I(A)$ if and only if $q\in I(A)$.

We use the standard subsequence form of well-quasi-ordering: every infinite sequence in a wqo has an infinite nondecreasing subsequence. Consequently, every infinite partially ordered wqo contains an infinite strictly increasing sequence, obtained by enumerating infinitely many distinct elements and passing to such a subsequence.

\begin{lemma}\label{lem:downward-universe}
Let $Q_0$ be a downward-closed suborder of $Q$. If $A$ is a countable $Q_0$-multiset and $B\meq A$ as $Q$-multisets, then every type occurring in $B$ belongs to $Q_0$.
\end{lemma}

\begin{proof}
Let $f$ witness $B\mle A$. For each $b\in\Omega_B$, we have $t_B(b)\leq t_A(f(b))\in Q_0$. Downward closure gives $t_B(b)\in Q_0$.
\end{proof}

The finiteness assertion in the next lemma is the countable case of \cite[Lemma~3.1]{KP24}; Krill and Pitz refer to \cite[Lemma~3.3]{BEEGHPT22} for its proof. We prove the additional exact decomposition because it is needed in the fixed-class count.

\begin{lemma}\label{lem:multiset-core}
For countable $Q$-multisets $A$ and $B$, we have $A\meq B$ if and only if $I(A)=I(B)$ and $F(A)\meq F(B)$. Moreover, $F(A)$ is finite.
\end{lemma}

\begin{proof}
Order the occurrence set $\Omega_A$ by $a\preceq_A b$ if $t_A(a)\leq t_A(b)$. This is a wqo, and the occurrences in $F(A)$ are exactly those that are not $\aleph_0$-embeddable in $(\Omega_A,\preceq_A)$. Hence \cite[Lemma~3.1]{KP24} gives that $F(A)$ is finite; see also \cite[Lemma~3.3]{BEEGHPT22}.

If $A\mle B$, then every $q\in I(A)$ lies below infinitely many distinct images in $B$, and therefore belongs to $I(B)$. Thus $A\meq B$ implies $I(A)=I(B)=:I$. Let $f$ witness $A\mle B$. If $a\in F(A)$ and $f(a)\notin F(B)$, then $t_B(f(a))\in I$. Since $I$ is downward closed and $t_A(a)\leq t_B(f(a))$, we obtain $t_A(a)\in I$, a contradiction. Hence $f$ restricts to $F(A)\mle F(B)$. The reverse embedding gives $F(A)\meq F(B)$.

Conversely, assume that $I(A)=I(B)=I$ and $F(A)\meq F(B)$. Put $A_0=A-F(A)$ and $B_0=B-F(B)$. Enumerate $A_0$ as $a_0,a_1,\ldots$, stopping if $A_0$ is finite. For each listed occurrence $a_n$, infinitely many occurrences of $B$ lie above $t_A(a_n)$, and only finitely many belong to $F(B)$. We may therefore choose distinct $b_n\in B_0$ with $t_A(a_n)\leq t_B(b_n)$. Thus $A_0\mle B_0$, and similarly $B_0\mle A_0$. Combining these embeddings with those of the exceptional parts gives $A\meq B$.
\end{proof}

\begin{lemma}\label{lem:finite-rigidity}
If $F$ and $G$ are finite $Q$-multisets, then $F\meq G$ if and only if every $\qeq$-class has the same multiplicity in $F$ and in $G$.
\end{lemma}

\begin{proof}
The reverse implication is immediate. Suppose that $F\meq G$, and let $f:F\mle G$ and $g:G\mle F$ witness the two embeddings. Since the occurrence sets are finite, both maps are bijections. The composition $h=g\circ f$ is a permutation of the occurrences of $F$. For every occurrence $a$ of $F$, we have $t_F(a)\leq t_G(f(a))\leq t_F(h(a))$. Following the finite $h$-cycle through $a$ gives $t_F(h(a))\leq t_F(a)$. Hence $t_F(a)\qeq t_G(f(a))$. Thus $f$ preserves every $\qeq$-class and the multiplicities agree.
\end{proof}

Diestel's domination order on subsets of a wqo \cite{Diestel01} is a useful set-valued precursor to the infinite part below, although it does not record multiplicities or occurrence-level injections. The replacement and increasing-chain constructions in the proof adapt ideas from \cite[Lemmas~2.2 and~2.3]{HPW24} and \cite[Lemmas~3.3 and~4.3]{AbdiDSC25}. The next theorem gives the exact cardinality of one fixed mutual-embedding class of occurrence multisets over an arbitrary wqo.

\begin{theorem}\label{thm:multiset-trichotomy}
Assume that $|Q|\leq2^{\aleph_0}$ and that every $\qeq$-class has size $1$, $\aleph_0$, or $2^{\aleph_0}$. For $q\in Q$, let $\theta(q)=|[q]_{\qeq}|$; this value is constant on each $\qeq$-class. Let $A$ be a countable $Q$-multiset and put $P(A)=I(A)/{\qeq}$. Define
\[
 g(A)=
 \begin{cases}
 2^{\aleph_0},&P(A)\text{ is infinite},\\
 2^{\aleph_0},&P(A)\text{ is finite and contains a nonrigid class},\\
 \aleph_0,&P(A)\text{ is finite, rigid, and not an antichain},\\
 1,&P(A)\text{ is a finite rigid antichain},
 \end{cases}
\]
and let $e(A)=\max\bigl(\{1\}\cup\{\theta(t_A(a)):a\in F(A)\}\bigr)$. If $\chi(A)$ is the number of exact multiset isomorphism types $B$ with $B\meq A$, then $\chi(A)=\max\{g(A),e(A)\}$. In particular, $\chi(A)\in\{1,\aleph_0,2^{\aleph_0}\}$.
\end{theorem}

\begin{proof}
By Lemma~\ref{lem:multiset-core}, every $B\meq A$ has the same set $I=I(A)$, and its exceptional part is mutually embeddable with $F(A)$. Lemma~\ref{lem:finite-rigidity} fixes the multiplicity of each $\qeq$-class in the exceptional part.

We first count the exact exceptional parts. Suppose that a class $E$ occurs $m\geq1$ times in $F(A)$. If $|E|=1$, there is one exact $m$-multiset from $E$. If $|E|$ is $\aleph_0$ or $2^{\aleph_0}$, there are exactly $|E|$ exact $m$-multisets from $E$. Only finitely many classes occur in $F(A)$, so the number of possible exceptional parts is $e(A)$.

It remains to count the multisets $C$ with $I(C)=I$ and $F(C)=\varnothing$. Put $P=I/{\qeq}$.

Suppose first that $P$ is finite and all its classes are rigid. Every maximal class of $P$ must occur $\aleph_0$ times in $C$: infinitely many occurrences lie above it, and maximality forces those occurrences to lie in the same class. Conversely, if every maximal class occurs $\aleph_0$ times and every occurrence has a type in $I$, then $I(C)=I$ and $F(C)=\varnothing$. Each nonmaximal class may occur $0,1,2,\ldots$, or $\aleph_0$ times. Hence there is one possibility when $P$ is an antichain and exactly $\aleph_0$ possibilities otherwise.

Suppose next that $P$ is finite and contains a nonrigid class $E$. Choose distinct exact types $e_*,e_0,e_1,\ldots$ in $E$. For every maximal class $M$ of $P$, choose a representative $r_M$, taking $r_E=e_*$ when $E$ is maximal. Let $R$ contain countably many copies of every $r_M$. Then $I(R)=I$ and $F(R)=\varnothing$. For $X\subseteq\mathbb N$, put $R_X=R+\{e_n:n\in X\}$. These multisets all have the same set $I$ and empty exceptional part, while different sets $X$ give nonisomorphic exact multisets. This is the occurrence-multiset form of the component-replacement construction in \cite[Lemma~2.2]{HPW24} and \cite[Lemma~3.3]{AbdiDSC25}. Thus there are $2^{\aleph_0}$ possibilities.

Finally, suppose that $P$ is infinite. By the standard strict-chain consequence of well-quasi-ordering, choose $x_0<x_1<\cdots$ in $P$, and put $p_n=x_{2n}$ and $u_n=x_{2n+1}$. Let $D$ be the set of classes occurring in $A-F(A)$. It is countable and cofinal in $P$: for every $q\in I$, some occurrence of $A-F(A)$ has type at least $q$, so its class lies above $[q]_{\qeq}$. Set $D'=(D\setminus\{p_n:n\in\mathbb N\})\cup\{u_n:p_n\in D\}$. The set $D'$ is still cofinal, because a removed $p_n$ is replaced by the larger class $u_n$. The strict chain also shows that no added $u_m$ equals any $p_n$, so $D'$ contains no $p_n$. Choose one exact representative from each class in $D'$, and let $R$ contain countably many copies of every chosen representative. Then $I(R)=I$ and $F(R)=\varnothing$. If $z_n$ represents $p_n$, put $R_X=R+\{z_n:n\in X\}$ for $X\subseteq\mathbb N$. Again, all $R_X$ have the same $I$ and empty exceptional part, and different sets $X$ give nonisomorphic multisets. This is the multiset analogue of the increasing-chain constructions in \cite[Lemma~2.3]{HPW24} and \cite[Lemma~4.3]{AbdiDSC25}. Hence there are $2^{\aleph_0}$ possibilities.

Every finite or countably infinite $Q$-multiset can be represented by a sequence in $Q\cup\{*\}$, where $*$ is a padding symbol. Thus the total number of such multisets is at most $(|Q|+1)^{\aleph_0}\leq2^{\aleph_0}$. The lower bounds above are therefore exact.

It remains to combine the two parts. Let $C$ satisfy $I(C)=I$ and $F(C)=\varnothing$, and let $E$ be a finite multiset whose $\qeq$-class multiplicities agree with those of $F(A)$. Every type in $E$ lies outside $I$, because $I$ is a union of $\qeq$-classes and the classes in $E$ are those of $F(A)$. Adding the finite multiset $E$ does not change which types have infinitely many upper occurrences. Hence $I(C+E)=I$ and $F(C+E)=E$. Conversely, every $B\meq A$ has the unique decomposition $(B-F(B))+F(B)$. The choices of the infinite and exceptional parts are therefore independent. Their product is the maximum of the two cardinals, which proves the formula.
\end{proof}

\section{Canonical decomposition}

Throughout this section, anchored trees are countable and small, and their branch hulls are nonempty.

\begin{definition}\label{def:planted-pieces}
Let $\mathcal Q_*$ be the quasi-order of exact isomorphism types of countable planted trees under $\plantm$. Let $(T,A)$ have canonical core $C=C(T,A)$. Every component $U$ of $T-C$ has a unique neighbour in $C$, denoted by $v(U)$. The \emph{planted piece} associated with $U$ is $P_U=(T[U\cup\{v(U)\}],v(U))$. For $v\in C$, let $\mathcal A_T(v)$ be the countable $\mathcal Q_*$-multiset whose occurrences are the pieces $P_U$ with $v(U)=v$ and whose exact types are their planted isomorphism types.
\end{definition}

The next statement is the branch-hull and planted-tree form of the standard kernel fact that every component outside $K(G)$ has smaller rank \cite[Section~2]{BBDS11}. The short proof records the effect of adjoining the planted root.

\begin{lemma}\label{lem:rank-drop}
If $\rho(T,A)=\alpha>0$, then every planted piece $P_U$ has branch rank below $\alpha$.
\end{lemma}

\begin{proof}
Let $B=B(T,A)$ and $C=C(T,A)$. The branching vertices of $P_U$ are exactly $\Br(T)\cap U$. If this set is empty, then the branch hull of $P_U$ consists only of its root and has rank $0$.

Assume that $\Br(T)\cap U\neq\varnothing$. The subtree $B\cap U$ is a component $D$ of $B-C$. Since $K(B)\subseteq C$, the component $D$ lies in a component of $B-K(B)$ and therefore has rank below $\alpha$ by the defining kernel property \cite[Section~2]{BBDS11}. The branch hull of $P_U$ is contained in the tree induced by $V(D)\cup\{v(U)\}$, which is a finite vertex extension of $D$. Lemma~\ref{lem:finite-extension} and rank monotonicity give $\rho(P_U)<\alpha$.
\end{proof}

\begin{lemma}\label{lem:piece-confinement}
Let $\phi:(T,A)\topm(S,A')$ be a model that maps $C(T,A)$ bijectively onto $C(S,A')$. If $U$ is a component of $T-C(T,A)$ attached at $v$, then there is a unique component $W$ of $S-C(S,A')$ attached at $\phi(v)$ such that $\phi$ restricts to a root-fixing model $P_U\plantm P_W$. Distinct components of $T-C(T,A)$ attached at $v$ are sent to distinct components of $S-C(S,A')$.
\end{lemma}

\begin{proof}
Write $C_T=C(T,A)$ and $C_S=C(S,A')$. No vertex of $U$ maps into $C_S$, because every vertex of $C_S$ is already the image of a vertex of $C_T$. No edge-model path of $P_U$ can contain another vertex of $C_S$ as an internal vertex, for the same reason. Hence the model of $P_U$ meets $C_S$ only at $\phi(v)$.

After deleting $\phi(v)$, the model of $P_U$ remains connected, so it lies in one component $W$ of $S-C_S$. The unique neighbour of $W$ in $C_S$ is $\phi(v)$, and the restricted model fixes the root. If two different source components attached at $v$ entered the same $W$, the model paths of their first edges would both use the unique edge from $\phi(v)$ into $W$. This contradicts the disjointness condition for a topological model.
\end{proof}

The next theorem is the anchored topological-minor analogue of the finite-kernel decomposition for ordinary embeddings in \cite[Lemma~4]{BBDS11}; compare also the leafless-core representation in \cite{AbdiCB23}. Lemma~\ref{lem:piece-confinement} supplies the additional separation needed for subdivisions and root-fixing pieces.

\begin{theorem}\label{thm:decomposition}
Let $(T,A)$ and $(S,A')$ be countable small anchored trees with nonempty branch hulls. Then $(T,A)\topeq(S,A')$ if and only if there is an anchor-preserving isomorphism $\sigma:C(T,A)\to C(S,A')$ such that $\mathcal A_T(v)\meq\mathcal A_S(\sigma(v))$ for every $v\in C(T,A)$. Moreover, $(T,A)\cong(S,A')$ if and only if $\sigma$ can be chosen so that all corresponding planted multisets are exactly isomorphic.
\end{theorem}

\begin{proof}
We follow the finite-kernel matching scheme of \cite[Lemma~4]{BBDS11}, using Corollary~\ref{cor:core-rigidity} and Lemma~\ref{lem:piece-confinement} for the topological adaptation. Suppose first that $(T,A)\topeq(S,A')$, and choose models $\phi:(T,A)\topm(S,A')$ and $\psi:(S,A')\topm(T,A)$. Their restrictions to the canonical cores are isomorphisms. Put $\sigma=\phi|_{C(T,A)}$ and $\tau=\psi|_{C(S,A')}$. Lemma~\ref{lem:piece-confinement} gives $\mathcal A_T(v)\mle\mathcal A_S(\sigma(v))$ for every core vertex $v$, and the reverse model gives $\mathcal A_S(w)\mle\mathcal A_T(\tau(w))$ for every target-core vertex $w$.

Let $\pi=\tau\sigma\in\Aut(C(T,A))$. Fix $v$ and let $m$ be the length of its finite $\pi$-orbit. Alternating the two local embeddings along this orbit gives
\[
 \mathcal A_S(\sigma(v))\mle\mathcal A_T(\pi(v))\mle\mathcal A_S(\sigma(\pi(v)))\mle\cdots\mle\mathcal A_T(\pi^m(v))=\mathcal A_T(v).
\]
Together with $\mathcal A_T(v)\mle\mathcal A_S(\sigma(v))$, this proves mutual embeddability at $v$.

Conversely, suppose that $\sigma$ and the local mutual embeddings are given. At each core vertex, use a multiset injection to match the source pieces with distinct target pieces, and choose a root-fixing model for each matched pair. The target pieces are disjoint outside their roots, so these models glue to the core isomorphism and give $(T,A)\topm(S,A')$. The reverse local embeddings give a model in the other direction.

A global isomorphism maps the canonical core onto the canonical core and induces exact multiset isomorphisms at each core vertex. Conversely, exact local multiset isomorphisms glue to the core isomorphism and give a global isomorphism.
\end{proof}

\begin{definition}\label{def:counting-parameters}
For a countable small anchored tree $(T,A)$, let $\Theta(T,A)$ be the number of anchored isomorphism types in its anchored topological-equivalence class. In particular, $\Theta(T,\varnothing)=|[T]|$. For a planted tree $P=(T,r)$, write $\Theta(P)=\Theta(T,\{r\})$. If $B(T,A)\neq\varnothing$ and $v\in C(T,A)$, let $\mathcal X_T(v)$ be the set of isomorphism types of countable $\mathcal Q_*$-multisets $M$ satisfying $M\meq\mathcal A_T(v)$, and put $\chi_T(v)=|\mathcal X_T(v)|$.
\end{definition}

\begin{corollary}\label{cor:finite-core-count}
Let $(T,A)$ be a countable small anchored tree with nonempty branch hull. If $\chi_T(v)\in\{1,\aleph_0,2^{\aleph_0}\}$ for every $v\in C(T,A)$, then $\Theta(T,A)=\max_{v\in C(T,A)}\chi_T(v)$.
\end{corollary}

\begin{proof}
Write $C=C(T,A)$ and $\mathcal X=\prod_{v\in C}\mathcal X_T(v)$. Since $C$ is finite, $|\mathcal X|=1$ when all factors are singletons; otherwise $|\mathcal X|$ is the largest factor.

Fix $x=(x_v)_{v\in C}\in\mathcal X$. Choose a representative multiset for every $x_v$, choose a countable planted tree representing each occurrence type, and attach these trees to the fixed core $C$. The resulting graph $T_x$ is a tree. Its anchored isomorphism type is independent of these choices: an exact multiset isomorphism pairs occurrences of the same planted isomorphism type, and the corresponding planted isomorphisms glue to the identity on $C$. Thus $x\mapsto[T_x,A]$ is well defined. The core is finite, every local occurrence set is countable, and every attached tree is countable; hence $T_x$ is countable.

The anchor remains a leaf. Indeed, if $A=\{r\}$ and $|C|>1$, then the unique neighbour of $r$ lies in $C$, so $\mathcal A_T(r)$ and hence $x_r$ are empty. If $C=\{r\}$, then $\mathcal A_T(r)$ has one occurrence, and mutual injections force $x_r$ to have one occurrence.

Gluing the identity on $C$ to the local embeddings in both directions gives $(T_x,A)\topeq(T,A)$. By Corollary~\ref{cor:smallness-invariant}, the tree $T_x$ is small and has nonempty branch hull. Apply Corollary~\ref{cor:core-rigidity} to the model $(T_x,A)\topm(T,A)$ that fixes the displayed copy of $C$. Its canonical core maps bijectively onto $C$. Every displayed vertex $c\in C$ is fixed by the model and is the unique preimage of $c$, so it belongs to $C(T_x,A)$. The two cores have the same finite order; therefore the displayed copy of $C$ is exactly $C(T_x,A)$.

Now let $(S,A')\topeq(T,A)$. By Theorem~\ref{thm:decomposition}, choose a core isomorphism from $C$ onto $C(S,A')$ and pull the exact local multisets of $S$ back to $C$. This gives a tuple $x\in\mathcal X$ for which $(T_x,A)\cong(S,A')$. Hence the construction maps $\mathcal X$ onto the anchored isomorphism types in the topological-equivalence class of $(T,A)$.

If $(T_x,A)\cong(T_y,A)$, then the isomorphism restricts to an element $\gamma\in\Aut(C,A)$ and sends $x_v$ to $y_{\gamma(v)}$ for every $v\in C$. For fixed $x$ and $\gamma$, the tuple $y$ is determined. Thus every anchored isomorphism type has at most $|\Aut(C,A)|$ preimages. The construction is finite-to-one. Its image has cardinality $|\mathcal X|$ when $\mathcal X$ is infinite, and cardinality $1$ when $|\mathcal X|=1$. This proves the formula.
\end{proof}

\section{Counting countable small trees}

We first treat branch rank $0$.

\begin{definition}\label{def:arms}
Let $(T,A)$ have finite nonempty branch hull. Then $C(T,A)=B(T,A)$. Every component of $T-C(T,A)$ is a finite path or a ray, attached to the core at an endvertex. Its planted piece is a \emph{path arm of length $n$} if it has $n$ vertices outside the root, and a \emph{ray arm} if it is infinite. For $v\in C(T,A)$, let $a_v(n)$ be the number of path arms of length $n$ and let $a_v(\infty)$ be the number of ray arms. For $k\geq1$, put $M_v(k)=a_v(\infty)+\sum_{n\geq k}a_v(n)$. Let $q(v)=\sup\{k\geq1:M_v(k)=\aleph_0\}$, with $q(v)=0$ if the set is empty and $q(v)=\infty$ if it is unbounded. Finally, let $Q_{\mathrm{arm}}$ be the set of exact planted isomorphism types of path arms and ray arms, ordered by $\plantm$.
\end{definition}

\begin{corollary}\label{cor:finite-hull}
Let $(T,A)$ be a countable small anchored tree with finite nonempty branch hull. Then
\[
 \Theta(T,A)=
 \begin{cases}
 1,&q(v)\leq1\text{ for every }v\in C(T,A),\\
 \aleph_0,&q(v)<\infty\text{ for every }v\text{ and }q(v)\geq2\text{ for some }v,\\
 2^{\aleph_0},&q(v)=\infty\text{ for some }v.
 \end{cases}
\]
\end{corollary}

\begin{proof}
Fix $v\in C(T,A)$. The order $Q_{\mathrm{arm}}$ is downward closed in the planted-tree order. Indeed, if $P\plantm Q$ and $Q$ is an arm, then $P$ has no branching vertex; since its root is a leaf, $P$ is also an arm. Lemma~\ref{lem:downward-universe} shows that every planted multiset mutually embeddable with $\mathcal A_T(v)$ uses only arm types.

The arm types form the rigid chain $1<2<3<\cdots<\infty$. A finite type $k$ belongs to $I(\mathcal A_T(v))$ exactly when $M_v(k)=\aleph_0$, and the ray type belongs to $I(\mathcal A_T(v))$ exactly when $a_v(\infty)=\aleph_0$. Hence $I(\mathcal A_T(v))/{\qeq}$ is empty or a singleton when $q(v)\leq1$, a finite non-antichain when $2\leq q(v)<\infty$, and infinite when $q(v)=\infty$. Theorem~\ref{thm:multiset-trichotomy} gives
\[
 \chi_T(v)=
 \begin{cases}
 1,&q(v)\leq1,\\
 \aleph_0,&2\leq q(v)<\infty,\\
 2^{\aleph_0},&q(v)=\infty.
 \end{cases}
\]
Corollary~\ref{cor:finite-core-count} now gives the stated formula.
\end{proof}

The transfinite organization below follows the Schmidt-rank induction used for ordinary twins in \cite[Lemma~5]{BBDS11} and the lower-type organization for rooted topological types in \cite[Section~4]{KP24}. The fixed-class cardinal input is Theorem~\ref{thm:multiset-trichotomy}.

\begin{theorem}\label{thm:small-recursion}
Let $(T,A)$ be a countable small anchored tree with nonempty branch hull. Then $\Theta(T,A)\in\{1,\aleph_0,2^{\aleph_0}\}$.

Suppose that $\rho(T,A)=\alpha>0$. Let $\mathcal Q_{<\alpha}$ be the set of isomorphism types of countable small planted trees of branch rank below $\alpha$, ordered by $\plantm$. For $v\in C(T,A)$, regard $\mathcal A_T(v)$ as a $\mathcal Q_{<\alpha}$-multiset and put $I_v=I(\mathcal A_T(v))$, $F_v=F(\mathcal A_T(v))$, and $P_v=I_v/{\qeq}$. For $p\in\mathcal Q_{<\alpha}$, let $\theta(p)=\Theta(P)$, where $P$ is any planted tree representing $p$; this is well defined. Then Theorem~\ref{thm:multiset-trichotomy} gives $\chi_T(v)$, and $\Theta(T,A)=\max_{v\in C(T,A)}\chi_T(v)$.

More explicitly, $\chi_T(v)=2^{\aleph_0}$ if and only if at least one of the following holds:
\begin{enumerate}[label=\textup{(\roman*)}]
\item $P_v$ is infinite;
\item $\theta(p)>1$ for some exact type $p\in I_v$;
\item $\theta(p)=2^{\aleph_0}$ for some exact type $p$ occurring in $F_v$.
\end{enumerate}
If none of these conditions holds, then $\chi_T(v)=\aleph_0$ if and only if $P_v$ is a finite non-antichain or $\theta(p)=\aleph_0$ for some exact type $p$ occurring in $F_v$. In all other cases, $\chi_T(v)=1$.
\end{theorem}

\begin{proof}
Following the rank-induction framework of \cite[Lemma~5]{BBDS11} and the lower-type organization of \cite[Section~4]{KP24}, we prove the cardinal conclusion and the recursion simultaneously for all countable small anchored trees with nonempty branch hull, by transfinite induction on the branch rank. Corollary~\ref{cor:finite-hull} is the rank-$0$ case. Let $\alpha>0$, assume both assertions for every smaller rank, and let $(T,A)$ have branch rank $\alpha$.

By Lemma~\ref{lem:rank-drop}, every planted piece in every $\mathcal A_T(v)$ has branch rank below $\alpha$, so it belongs to $\mathcal Q_{<\alpha}$. There are at most $2^{\aleph_0}$ isomorphism types of countable rooted trees, and Corollary~\ref{cor:planted-wqo} shows that $\mathcal Q_{<\alpha}$ is a wqo.

The suborder $\mathcal Q_{<\alpha}$ is downward closed. Indeed, suppose that $P'\plantm P$ and $P\in\mathcal Q_{<\alpha}$. The tree $P'$ is countable because the model is injective on vertices. It is small by Corollary~\ref{cor:smallness-downward}, and Lemma~\ref{lem:hull-monotone} together with rank monotonicity gives $\rho(P')\leq\rho(P)<\alpha$. Thus $P'\in\mathcal Q_{<\alpha}$. Lemma~\ref{lem:downward-universe} therefore shows that Theorem~\ref{thm:multiset-trichotomy} counts the full local class $\mathcal X_T(v)$.

If two elements of $\mathcal Q_{<\alpha}$ are mutually $\plantm$-contained, rank monotonicity gives equality of their branch ranks. Hence the $\qeq$-class of $p\in\mathcal Q_{<\alpha}$ is the full countable planted topological-equivalence class of any representative $P$ of $p$. By the induction hypothesis, its cardinality is $\Theta(P)=\theta(p)\in\{1,\aleph_0,2^{\aleph_0}\}$. All hypotheses of Theorem~\ref{thm:multiset-trichotomy} are now satisfied. It gives the stated alternatives for every $\chi_T(v)$. Corollary~\ref{cor:finite-core-count} then yields $\Theta(T,A)=\max_{v\in C(T,A)}\chi_T(v)$ and completes the induction.
\end{proof}

\begin{proof}[Proof of Theorem~\ref{thm:main}]
Let $T$ be a countable small tree. Suppose first that $\Br(T)=\varnothing$, and let $S\topeq T$. Since $S\topm T$, every branching vertex of $S$ would map to a branching vertex of $T$; hence $\Br(S)=\varnothing$. Thus both $S$ and $T$ are finite paths, rays, or double rays. No infinite path is a topological minor of a finite path, and a double ray is not a topological minor of a ray. Within the finite paths, mutual topological containment forces the same order; every subdivision of a ray or of a double ray has the same respective type. Hence $S\cong T$, so $|[T]|=1$.

Suppose that $\Br(T)\neq\varnothing$. Lemma~\ref{lem:small-hull} shows that $B(T,\varnothing)$ is a nonempty countable rayless tree. Theorem~\ref{thm:small-recursion}, applied with empty anchor, gives $|[T]|=\Theta(T,\varnothing)\in\{1,\aleph_0,2^{\aleph_0}\}$.
\end{proof}

\begin{proof}[Proof of Corollary~\ref{cor:countable}]
If $T$ is small, apply Theorem~\ref{thm:main}. If $T$ is large, Bruno and Szeptycki \cite[Theorem~3.1]{BS23} give at least $2^{\aleph_0}$ pairwise nonisomorphic topological twins. Every $S\topeq T$ is countable because a model $S\topm T$ is injective on vertices. There are at most $2^{\aleph_0}$ isomorphism types of countable trees, so $|[T]|=2^{\aleph_0}$.
\end{proof}

\begin{corollary}\label{cor:locally-countable}
Every locally countable tree $T$ satisfies $|[T]|\in\{1,\aleph_0,2^{\aleph_0}\}$.
\end{corollary}

\begin{proof}
Fix $r\in V(T)$ and let $L_n$ be the set of vertices at distance $n$ from $r$. Each $L_n$ is countable by induction, and $V(T)=\bigcup_{n<\omega}L_n$ is countable. Corollary~\ref{cor:countable} applies.
\end{proof}

\begin{remark}
Theorem~\ref{thm:small-recursion} gives a well-founded recursive description. At each core vertex, $I_v$ records the lower-rank planted types that lie below infinitely many attached pieces, while the finite multiset $F_v$ records the remaining pieces. The recursion ends at a finite branch hull, where Corollary~\ref{cor:finite-hull} gives the path-and-ray criterion.
\end{remark}

\section*{Declaration of competing interest}

The authors declare that they have no known competing financial interests or personal relationships that could have appeared to influence the work reported in this paper.

\section*{Data availability}

No data were used for the research described in this article.

\section*{Declaration on the use of generative AI}

During the preparation of this work, the authors used OpenAI's ChatGPT to discuss possible proof strategies and to improve the presentation. The authors subsequently reviewed and verified every argument and take full responsibility for the mathematical content of the manuscript.

\end{document}